\documentclass[a4paper,11pt]{article}

\usepackage{amsmath}
\usepackage{amsfonts}
\usepackage{amssymb}
\usepackage[english]{babel}
\usepackage{graphicx}
\usepackage{amsthm}
\newtheorem{thm}{Theorem}[section]

\newtheorem{lem}[thm]{Lemma}

\newtheorem{rem}[thm]{Remark}
\newtheorem{ex}[thm]{Example}

\theoremstyle{definition}

\numberwithin{equation}{section}

\newcommand{\Z}{\mathbb{Z}}

\newcommand{\ol}{\overline}

\begin{document}

\title{$3$-pyramidal Steiner Triple Systems
\footnote{Research performed within the activity of INdAM--GNSAGA with the
financial support of the italian Ministry MIUR,
project ``Combinatorial Designs, Graphs and their Applications''}}

\author{Marco Buratti \thanks{Dipartimento di Matematica e Informatica, Universit\`a di Perugia, via Vanvitelli 1 - 06123 Italy, email: buratti@dmi.unipg.it}\quad\quad
Gloria Rinaldi
\thanks{ Dipartimento di Scienze e Metodi dell'Ingegneria,
Universit\`a di Modena e Reggio Emilia, via Amendola, Italy. email: gloria.rinaldi@unimore.it}\quad\quad
Tommaso Traetta
\thanks{Department of Mathematics, Ryerson University, Toronto (ON) M5B 2K3, Canada,
email: tommaso.traetta@ryerson.ca, traetta.tommaso@gmail.com}}

\maketitle
\bigskip
\begin{abstract}
\noindent
A design is said to be $f$-pyramidal
when it has an automorphism group which fixes $f$ points and acts sharply transitively
on all the others. The problem of establishing the set of values of $v$ for which there exists an\break $f$-pyramidal
Steiner triple system of order $v$ has been deeply investigated in the case $f=1$ but it remains open for a special class of values of $v$.
The same problem for the next possible $f$, which is $f=3$, is here completely solved: there exists a $3$-pyramidal Steiner triple system of order $v$ if and only
if $v\equiv7,9,15$ (mod $24$) or $v\equiv3,19$ (mod 48).
\end{abstract}


\bigskip\bigskip\bigskip\bigskip\bigskip\bigskip\bigskip\bigskip\bigskip\bigskip\bigskip\bigskip\bigskip\bigskip\bigskip\bigskip\bigskip

\eject
\normalsize
\section{Introduction}
\noindent
A Steiner triple system of order $v$, briefly STS$(v)$, is a pair $(V,{\cal B})$ where $V$ is a set of $v$ points and
$\cal B$ is a set of 3-subsets ({\it blocks} or {\it triples}) of $V$ with the property that any two distinct points are contained in
exactly one block. Apart from the trivial case $v=0$ in which both $V$ and $\cal B$ are empty, it is well known that a STS$(v)$ exists if
and only if $v\equiv 1$ or 3 (mod 6).
For general background on STSs we refer to \cite{CR}.

Steiner triple systems having an automorphism with a prescribed property or an automorphism group
with a prescribed action have drawn much attention since a long time.  It was proved by Peltesohn \cite{P}
that a STS$(v)$ with an automorphism cyclically permuting all points,
briefly a {\it cyclic} STS$(v)$, exists for any possible $v$ but $v\neq9$.
The existence question for a STS$(v)$ with an involutory automorphism fixing exactly one point,
briefly a {\it reverse} STS$(v)$, has been settled by means of three different contributions of Doyen \cite{D},
Rosa \cite{R} and Teirlinck \cite{T}; it exists if and only if $v\equiv1, 3, 9, 19$ (mod 24).
In \cite{PR} Phelps and Rosa proved that there exists a STS$(v)$ with an automorphism cyclically permuting all but
one point, briefly a {\it $1$-rotational} STS$(v)$, if and only if $v\equiv3$ or 9 (mod 24).

Note that a cyclic or 1-rotational STS may be viewed as a STS with a cyclic automorphism group
acting sharply transitively on all points or all but one point, respectively. Thus one may ask, more generally, for
a STS with an automorphism group $G$ having the same kind of action without the request that $G$ be cyclic.

Speaking of a regular STS$(v)$ we mean a STS$(v)$ for which there is at least one group
$G$ acting sharply transitively on the points. Also, speaking of a 1-rotational STS$(v)$
we mean a STS$(v)$ with at least one automorphism group acting regularly on all but one points.

The only STS(9), which is the point-line design
associated with the affine plane over $\Z_3$, is clearly regular under $\Z_3^2$.
Thus, in view of Peltesohn's result, there exists a regular STS$(v)$ for any admissible $v$.

The problem of determining the set of values of $v$ for which there exists a 1-rotational
STS$(v)$ (under some group) has been deeply investigated in \cite{B,BBRT}. Such a STS is necessarily
reverse so that $v$ must be congruent to 1, 3, 9, or 19 (mod 24). For the cases $v\equiv3$ or 9 (mod 24)
the existence clearly follows from the result by Phelps and Rosa. In the case $v\equiv19$ (mod 24) we do not have
existence only when $v=6PQ+1$ with $P=1$ or a product of pairwise distinct primes congruent to 5 (mod 12) and
with $Q$ a product of an odd number of pairwise distinct primes congruent to 11 (mod 12).
The most difficult case is $v\equiv1$ (mod 24) where the existence remains still uncertain only when all
the following conditions are simultaneously satisfied: $v=(p^3-p)n+1\equiv1$ (mod 96) with $p$ a prime;
$n\not\equiv0$ (mod 4); the odd part of $v-1$ is square-free and without prime factors $\equiv1$ (mod 6).

Of course one might consider the more specific problems of determining all groups $G$ for which
there exists a STS which is regular under $G$ and all groups $G$ for which there exists
a STS which is 1-rotational under $G$. These problems appear at this moment quite hard.
The same problems can be relaxed by asking for which $v$ there is a STS$(v)$ which is regular
(1-rotational) under a group belonging to an assigned general class.
For instance, the first author proved that there exists a STS$(v)$ which is 1-rotational
under some abelian group if and only if either $v\equiv3, 9$ (mod 24) or $v\equiv1,19$ (mod 72).
Also, Mishima \cite{M} proved that there exists a STS$(v)$ which is regular
under a dicyclic group if and only if $v\equiv9$ (mod 24).

We have quoted only the results which are closer to the problem that we are going to study in this paper. Indeed
the literature on Steiner triple systems and their automorphism groups is quite wide. For example, results
concerning the {\it full} automorphism group of a STS have been obtained by Mendelsohn \cite{M} and Lovegrove \cite{L}.

Now we want to consider the problem of determining the set of values of $v$
for which there exists a STS$(v)$ with an automorphism group fixing $f$
points and acting sharply transitively on the other $v-f$ points. Such a STS will
be called {\it $f$-pyramidal}.
Of course the cases $f=0$ and $f=1$ correspond,
respectively, to the {\it regular} and 1-{\it rotational} STSs discussed above.
It is natural to study the next possible case that is $f=3$. This is because, as we are going to see in the next lemma,
the fixed points of an $f$-pyramidal STS$(v)$ form a subsystem of order $f$ so that, for $f\neq0$,
we have $f\equiv1$ or 3 (mod 6).

\begin{lem}\label{f}
A necessary condition for the existence of an $f$-pyramidal STS$(v)$ is that
$f=0$ or $f\equiv1, 3$ $($mod $6)$.
\end{lem}
\begin{proof}
Let $(V,{\cal B})$ be an $f$-pyramidal STS$(v)$
under the action of a group $G$. Assume, w.l.o.g., that $G$ is additive and
let $F = \{\infty_1,\dots,\infty_f\}$ be the set of points fixed by $G$. Obviously
$G$ has order $v-f$, the set of points $V$ can be identified with $F \ \cup \ G$,
and the action of $G$ on $V$ can be identified with the addition on the right with the assumption that
$\infty_i+g=\infty_i$ for each $\infty_i\in F$ and each $g\in G$.
If a block $B\in{\cal B}$ contains two distinct fixed points, say $\infty_i$ and $\infty_j$,
then $B+g=B$ for every $g\in G$ otherwise $\cal B$ would have two
distinct blocks, $B$ and $B+g$, passing through the two points $\infty_i$ and $\infty_j$.
So, the third vertex of $B$ is also fixed by $G$. It easily follows that all blocks of $\cal B$
contained in $F$ form a STS$(f)$ so that we have $f=0$ or $f\equiv1, 3$ (mod 6).
\end{proof}

The main result of this paper is a complete solution to the existence problem
for a 3-pyramidal STS$(v)$.

\begin{thm}\label{main}
There exists a $3$-pyramidal STS$(v)$ if and only if $v\equiv7, \ 9, \ 15$ $($mod $24)$
or $v\equiv 3, \ 19$ $($mod $48)$.
\end{thm}

The ``if part" of this theorem will be proved in Section 3 which therefore will give non-existence
results: for $v\equiv1, \ 13, \ 21$ (mod 24) or $v\equiv27, \ 43$ (mod 48) there is no 3-pyramidal STS$(v)$.
The ``only if part" will be proved in Section 4 where we will give an explicit
construction of a 3-pyramidal STS$(v)$ whenever
$v\equiv7, 9, 15$ $($mod $24)$ or $v\equiv 3, 19$ $($mod $48)$.

First, in the next section, we have to translate our problem into algebraic terms: any
$f$-pyramidal STS is completely equivalent to a suitable {\it difference family}.

\vskip 0.5truecm

\section{Difference families and pyramidal STSs}

As a natural generalization of the concept of a {\it relative difference set} \cite{Pott}, the first author
introduced  \cite{recursive} {\it difference families} in a group $G$ relative to a subgroup of $G$  or,
even more generally \cite{B2}, relative to a {\it partial spread} of $G$.
By a partial spread of a group $G$ one means a set $\Sigma$ of subgroups of $G$ whose mutual intersections
are all trivial. One omits the attribute ``partial" in the special case that the subgroups of $\Sigma$
cover all $G$. Let $\Sigma$ be a partial spread of an additively written group $G$, let $\cal F$ be a set of $k$-subsets of $G$, and let
$\Delta{\cal F}$ be the list of all possible differences $x-y$ with $(x,y)$ an ordered pair of distinct elements
of a member of $\cal F$. One says that $\cal F$ is a $(G,\Sigma,k,1)$-difference family (DF) if every group element
appears 0 or 1 times in $\Delta{\cal F}$ according to whether it belongs or does not belong to
some member of $\Sigma$, respectively.
We say that $\Sigma$ is of {\it type} $\{d_1^{e_1}, \dots, d_n^{e_n}\}$ if this is the multiset  (written
in ``exponential" notation) of the orders of all subgroups
belonging to $\Sigma$ and we speak of a $(G,\{d_1^{e_1}, \dots, d_n^{e_n}\},k,1)$-DF.

It is obvious that any STS$(v)$ is $v$-pyramidal under the trivial group. The following theorem explains how to
construct an $f$-pyramidal STS$(v)$ with $f<v$. It generalizes Theorem 1.1 in [3] which corresponds to the case $f=1$.

\begin{thm}\label{DF}
There exists an $f$-pyramidal STS$(v)$ with $f<v$ if and only if
there exists a $(G,\{2^f,3^e\},3,1)$-DF for a suitable group $G$ of order $v-f$ with exactly $f$ involutions,
and a suitable integer $e$.
\end{thm}
\begin{proof}
($\Longrightarrow$). Let $(V,{\cal B})$ be an $f$-pyramidal STS$(v)$
under an additive group $G$. We can assume that $V=F \ \cup \ G$ with $F$ and the action of $G$ on $V$ defined as in Lemma \ref{f}.
For $1\leq i\leq f$, let $B_i=\{\infty_i,0,x_i\}$ be the block of $\cal B$ containing the points $\infty_i$ and $0$.
We have $B_i-x_i=\{\infty_i,-x_i,0\}$ so that both $B_i$ and $B_i-x_i$ contain the points $\infty_i$ and $0$.
It necessarily follows that $B_i-x_i=B_i$, hence $-x_i=x_i$ which means that $x_i$ is an involution.
Conversely, if $x$ is an involution of $G$ and $B=\{0,x,y\}$ is the block through $0$ and $x$, then
$B+x=\{x,0,y+x\}$ would also contain $0$ and $x$ so that $B+x=B$. This means that $y+x=y$
and this is possible only if $y\in F$. Hence $y=\infty_i$ and $x=x_i$ for a suitable $i$.
We conclude that $\{x_1,...,x_f\}$ is the set of all involutions of $G$.

 Let $\cal F$ be  a complete  system of representatives for the $G$-orbits on the blocks of $\cal B$
 with trivial $G$-stabilizer. Reasoning as in the ``if part" of Theorem 2.2 in \cite{B}, one can see that
$\cal F$ is a $(G,\Sigma,3,1)$-DF where
$\Sigma$ is the partial spread of $G$ consisting of all 2-subgroups $\{0,x_i\}$ ($i=1,\dots,f$) of $G$
and all 3-subgroups of $G$ belonging to $\cal B$.

($\Longleftarrow$). Now assume that $f\equiv1$ or 3 (mod 6) and that
$\cal F$ is a $(G,\Sigma,3,1)$-DF with $G$ a group of order $v-f$ having
exactly $f$ involutions and with $\Sigma$ a partial spread of $G$ of type $\{2^f,3^e\}$.

Take an $f$-set $F=\{\infty_1,\dots,\infty_f\}$ disjoint with $G$ and let $(F,{\cal B}_{\infty})$
be any STS$(f)$ (which exists because we assumed that $f\equiv1$ or 3 (mod 6)).
For $i=2, 3$, let $\Sigma_i$ be the set of
subgroups of order $i$ belonging to $\Sigma$.
Set $\Sigma_2=\{S_i \ |  \ 1\leq i\leq f\}$ and $\Sigma_2^+=\{S_i \ \cup \ \{\infty_i\} \ | \ 1\leq i\leq f\}$.
Then, as in the ``only if part" of Theorem 2.2 in \cite{B}, one can see that $$\Sigma_2^+ \ \cup \ \Sigma_3 \ \cup \ \ {\cal F} \ \cup \ {\cal B}_\infty$$
is a complete system of representatives for the block-orbits of a 3-pyramidal STS$(v)$ under the action of $G$ on $F \ \cup \ G$ defined as in the
proof of Lemma \ref{f}.   \end{proof}

\begin{rem}\label{rem}
{\rm Considering that ``cyclic STS" means ``0-pyramidal STS under the cyclic group",
as a very special case of the above theorem we have the well known fact
that any cyclic STS$(6n+1)$ is equivalent to a $(\Z_{6n+1},\{1\},3,1)$-DF
and that any cyclic STS$(6n+3)$ is equivalent to a $(\Z_{6n+3},\{3\},3,1)$-DF.}
\end{rem}

\begin{ex}
The empty-set clearly is a $(\Z_2^n,\{2^{2^n-1}\},3,1)$-DF since every non-zero element of $\Z_2^n$ is an involution. It is not difficult to see that
the associated $(2^n-1)$-pyramidal STS$(2^{n+1}-1)$ is the point-line design of the
$n$-dimensional projective geometry over $\Z_2$.
\end{ex}

Let $\mathbb{D}_{2n}$ be the dihedral group of order $2n$, namely the group
with defining relations $\mathbb{D}_{2n}= \langle x, y \ | \ y^2=x^n=1;  \ yx=x^{-1}y\rangle$.
We give here an example of a STS$(3f)$ which is $f$-pyramidal under $\mathbb{D}_{2f}$.

\begin{ex}
{\rm
Let $f\equiv1$ or 3 (mod 6) but $f\neq9$.
Let $\phi: \Z_f\longrightarrow \mathbb{D}_{2f}$ be the group monomorphism defined by $\phi(i)=x^i$ for each $i\in\Z_f$.
The hypotesis on $f$ guarantees, in view of Peltesohn's result, that there exists a cyclic STS$(f)$. Thus, by Remark \ref{rem},
there exists a $(\Z_f,\{1\},3,1)$-DF or a $(\Z_f,\{3\},3,1)$-DF $\cal F$
according to whether $f\equiv1$ or $3$ (mod 6), respectively.
It is then obvious that $\{\phi(B) \ | \ B\in{\cal F}\}$ is a $(\mathbb{D}_{2f},\{2^f,3^e\},3,1)$-DF with $e=0$ or $1$, respectively.

Thus there exists an $f$-pyramidal STS$(3f)$ under the dihedral group $\mathbb{D}_{2f}$
for any $f\equiv1$ or $3$ $($mod $6)$ but $f\neq9$.}
\end{ex}

If, in the above example, we put $f=3$, we obtain a representation of the affine plane of order 3 as a 3-pyramidal STS(9) under $\mathbb{D}_{6}$.
In this case the difference family $\cal F$ is empty because it is relative to a spread which is not partial; its subgroups $\{1,y\}$, $\{1,xy\}$, $\{1,x^2y\}$
and $\{1,x,x^2\}$ cover indeed all elements of $\mathbb{D}_{6}$. Following the instructions of the ``only if part" of Theorem \ref{DF} the blocks of the STS(9) are:

$$\begin{matrix}
\{\infty_1,\infty_2,\infty_3\}, && \{1,x,x^2\}, && \{y,xy,x^2y\},\cr
\{\infty_1,1,y\}, && \{\infty_1,x,x^2y\}, && \{\infty_1,x^2,xy\},\cr
\{\infty_2,1,xy\},&& \{\infty_2,x,y\}, && \{\infty_2,x^2,x^2y\},\cr
\{\infty_3,1,x^2y\}, && \{\infty_3,x,xy\},&& \{\infty_3,x^2,y\}.
\end{matrix}$$

\medskip
In Section 4 we will make use of $\mathbb{D}_{6}$ again, for the construction of a 3-pyramidal STS$(24n+9)$ under
$\mathbb{D}_{6}\times \Z_{4n+1}$.

\section{The ``if part"}

In this section we determine the values of $v$ for which a 3-pyramidal STS$(v)$ cannot exist,
we namely prove the ``if part" of the main result Theorem \ref{main}.
For this, we need two lemmas about elementary group theory.

\begin{lem}\label{preTom}
If $G$ is a group of order $24n+18$,
then $G$ has a subgroup of index $2$.
\end{lem}
\begin{proof}
It is well known that a group of order twice an odd number has a subgroup of
index 2. See, for example, \cite[Exercise 262]{R}.
\end{proof}

The next lemma makes use of the so-called ``Burnside normal $p$-complement theorem''
which is here recalled (see \cite[Theorem 6.17]{R}).
\begin{thm}\label{Burnside}
Let $P$ be a Sylow $p$-subgroup of a finite group $G$.
If $C_G(P) = N_G(P)$, then $P$ has a normal complement in $G$.
%, that is, $G$ has a normal subgroup $K$ with the property $G =PK$.
%
%  If a Sylow $p$-subgroup $P$ of a finite group $G$ lies in the center of its normalizer in $G$,
%  then $P$ has a normal complement in $G$.
\end{thm}

\begin{lem}\label{Tom}
If $G$ is a group of order $16n+8$ containing exactly $3$ involutions, then $G$ has a subgroup
of index $2$ containing exactly one involution.
\end{lem}
\begin{proof}
Let $j_1,j_2$, and $j_3$ be the three involutions of $G$ and let
$H=\langle j_1, j_2, j_3 \rangle$
be the group they generate. We point out that $H$ is a normal subgroup of $G$ since
it is generated by all the elements of order $2$.
Now, let $P$ be a $2-$Sylow subgroup of $G$. As $P$ has order $8$ and $G$ contains exactly three involutions,
by taking into account the classification of the groups of order $8$,
we have three possibilities:
namely $P$ is either isomorphic to the group $\Bbb Z_4 \times \Bbb Z_2$ or
$P$ contains exactly one involution,
i.e., it is $P\simeq\Bbb Z_8$ or $P\simeq Q_8$ (the quaternion group of order $8$).

First of all we prove that it is necessarily $P\simeq\Bbb Z_4\times \Bbb Z_2$.
By contradiction, assume that $P$ contains exactly one involution.
It is known that, in general, any two involutions of $G$ generate a dihedral group; also,
  a dihedral group of order $2h$ contains at least $h$ involutions. Therefore, either
 $\langle j_1, j_2\rangle \simeq \mathbb{D}_4 \simeq \Z_2\times\Z_2$ or $\langle j_1, j_2\rangle \simeq \mathbb{D}_6$.
In both cases, $\langle j_1, j_2\rangle$ has three involutions, hence
$j_3 \in \langle j_1, j_2\rangle$ and $H=\langle j_1, j_2\rangle$.
Since $P$ contains just one involution, it is necessarily $H=\mathbb{D}_6$,
otherwise $P$  should contain a subgroup isomorphic to $\Bbb Z_2\times \Bbb Z_2$.

Let $T$ be the subgroup of $H$ of order $3$.
As usual, we denote by $N_G(T)$ and $C_G(T)$ the normalizer and the centralizer
of $T$ in $G$, respectively.
By the $N/C$-theorem, the quotient group
$N_G(T)/C_G(T)$ is isomorphic to a subgroup of $Aut(T)\simeq \Z_2$.
Since $T$ is the unique subgroup of $H$ of order $3$,
$T$ is characteristic in $H$ and then it is normal in $G$, that is $N_G(T)=G$.
Therefore, either $C_G(T) = G$ or $C_G(T)$ is a subgroup of $G$ of index $2$.
In both cases,  $C_G(T)$ is a normal subgroup of $G$ of even order.
The three involutions of $G$ are pairwise conjugate, hence,
they are contained in any normal subgroup of $G$ of even order and then in $C_G(T)$.
It then follows that $T$ is central in $H$
contradicting the fact that $H$ is dihedral.

We conclude that
$P \simeq \Z_4 \times \Z_2$; in particular, $P$ is abelian and contains a subgroup $Q\simeq \Z_4$; also, $P\supseteq H \simeq \Z_2 \times\Z_2$.
  To prove the assertion, we need to show that $G$ has a normal subgroup $O$ of order $2n+1$.
In fact, the semidirect product of $O$ and $Q$ will be
a subgroup of $G$ of index $2$ containing exactly one involution.

  Recall that $|G:C_G(j_i)|$ is the size of $cl(j_i)$, the conjugacy class of $j_i$ in $G$,
  which cannot exceed the total number of involutions in $G$ hence, $|G:C_G(j_i)|\leq 3$ for any $i=1,2,3$.
  On the other hand, $P$ is abelian hence, $P\leq C_G(j_i)$ for any $i=1,2,3$. It then follows that
  either $|G:C_G(j_i)| = 1$ for any $i=1,2,3$ or $|G:C_G(j_i)| = 3$ for any $i=1,2,3$.

  We first deal with the former case in which all involutions of $G$ are central.
  Since $G/H$ has order $2d$, $d$ odd,  then it has a subgroup of index $2$ (see for example  \cite[Exercise 262]{R}).
  In other words,
  there exists a normal subgroup $N$ of $G$ of index $2$ which contains $H$.
  Since $H$ is a central $2$-Sylow subgroup of $N$,
  by Theorem \ref{Burnside},
  $H$ has a normal complement $O$ in $N$; in particular, $O$ has order $2n+1$.
Now, suppose the existence of
   another subgroup $K$ of $N$ of order $2n+1$. Then $K\ \cap\ O$ is normal in $K$
   and the quotient $Q_1=K/(K\ \cap\ O)$ should be isomorphic
   to $Q_2=(K+O)/O$. However, $Q_1$ has odd order,
%   (since it is the quotient of a group of odd order),
   while $Q_2$ has even order.
%   (since it is contained in $N/O$ which has order $4$).
   Therefore, $O$ is the only subgroup of
  $N$ of order $2n+1$, hence it is normal in $G$.

  We finally consider the case $|G:C_G(j_i)| = 3$ for any $i=1,2,3$.
  By the N/C-theorem,
  $G/C_G(H)$ is isomorphic to a subgroup of $Aut(H)\simeq \mathbb{D}_6$.
  Since $P \leq C_G(H) \leq C_G(j_i)$, we have that $C_G(H) = C_G(j_i)$ for any $i=1,2,3$.
  Now note that $C_G(H)$ satisfies the same assumption as $G$ and all involutions are central
  in $C_G(H)$. Therefore, we can proceed as in the previous case to show that there is
   a normal subgroup $\Omega$ of $C_G(H)$ of order
   $\frac{2n+1}{3}$.
  As before, it is not difficult to check that $\Omega$ is the only subgroup of $C_G(H)$ of order $\frac{2n+1}{3}$, therefore it is normal in $G$.

  Set now $\ol{G} = G/\Omega$ and $\ol{P} = C_G(H)/\Omega$. Note that  $\ol{P}\simeq P$,
  $\ol{P}$ is normal in $\ol{G}$ (i.e., $N_{\ol{G}}(\ol{P})=\ol{G}$)
  and $\ol{G}/\ol{P}$ has order $3$. Also, since $\ol{P}$ is abelian, then
  $\ol{P}\leq C_{\ol{G}}(\ol{P})$ hence, $|\ol{G}:C_{\ol{G}}(\ol{P})|=1$ or $3$.
  Considering that, by the N/C-theorem, $\ol{G}/C_{\ol{G}}(\ol{P})$ is isomorphic to a subgroup of
  $Aut(\ol{P})$, and that $Aut(\ol{P})$ has order $8$, we then have that
  $|\ol{G}:C_{\ol{G}}(\ol{P})|=1$. This means that $\ol{P}$ is central in $\ol{G}$. Hence,
  $\ol{G}$ is the direct product of $\ol{P}$ by a normal subgroup $\ol{O}$ of order $3$ which is the
  quotient by $\Omega$ of a normal subgroup $O$ of $G$ of order $2n+1$.
\end{proof}

\begin{thm}\label{non-existence}
There is no $3$-pyramidal STS(v) in each of the following cases:
\begin{itemize}
\item[$(i)$] $v\equiv 1$ (mod $24$);
\item[$(ii)$] $v\equiv 13$ (mod $24$);
\item[$(iii)$] $v\equiv 21$ (mod $24$);
\item[$(iv)$] $v\equiv 27$ (mod $48$);
\item[$(v)$] $v\equiv 43$ (mod $48$).
\end{itemize}
\end{thm}
\begin{proof}
Cases (i)-(ii). Suppose that there exists a STS$(v)$ with $v=24n+1$ or $v=24n+13$
which is $3$-pyramidal under a group $G$.
Therefore $G$ is a group of order $24n-2$ or $24n+10$ with exactly three subgroups of order 2.
Now note that these subgroups are precisely the $2$-Sylow subgroups of $G$ since the order of $G$
is divisible by 2 but not by 4. Hence  $n_2(G)$, the number of $2$-Sylow subgroups of $G$, is $3$.
By the third Sylow theorem $n_2(G)$ should be also a divisor of ${|G|\over2}$. We conclude that
3 should be a divisor of $|G|$ which is clearly false.

Case (iii). Now assume that there exists a 3-pyramidal STS$(24n+21)$.
Then there exists a $(G,\{2^3,3^e\},3,1)$-DF $\cal F$ for a suitable group $G$ of order $24n+18$
with exactly 3 involutions and a suitable $e\geq0$.

By Lemma \ref{preTom} there is a subgroup $S$ of $G$ of index $2$. Note that an element of $G$
has odd or even order according to whether it is in $S$ or not, respectively.
Thus, in particular, the three involutions of $G$ are all contained in $G\setminus S$
while every subgroup of $G$ of order 3 is contained in $S$.
Thus the subset of $G\setminus S$ which is covered by $\Delta{\cal F}$ has size $|G\setminus S| -3=12n+6$.

If $B$ is any block of ${\cal F}$ having some differences in $G\setminus S$, then
it necessarily has two points lying in distinct cosets of $S$ in $G$. Thus,
up to translations, we have $B=\{0,s,t\}$ with $s\in S$ and $t\in G\setminus S$.
It follows that $\Delta B \ \cap \ (G\setminus S)=\{\pm t,\pm(s-t)\}$, hence $\Delta B$ has exactly four
elements in $G\setminus S$.

>From the above two paragraphs we conclude that $12n+6$ should be divisible by 4 which is clearly absurd.

Cases (iv)-(v). Assume that there exists a 3-pyramidal STS$(v)$ with $v\equiv27$ or $43$ (mod 48).
Thus there exists a $(G,\{2^3,3^e\},3,1)$-DF $\cal F$ for a suitable group $G$ of order $48n+24$ or $48n+40$
with exactly three involutions and where $e\geq0$ in case (iv)
or $e=0$ in case (v).
Note that in both cases we have $|G|\equiv8$ (mod 16) so that, by Lemma \ref{Tom},
$G$ has a subgroup $S$ of index $2$ containing exactly one
involution. Thus the subset of $G\setminus S$ which is covered by $\Delta{\cal F}$ has size $|G\setminus S|-2=24n+10$ or
$24n+18$.
Then, reasoning as in case (iii), $24n+10$ or $24n+18$ should be divisible by 4 which is absurd.
\end{proof}

\section{The ``only if part"}

For proving the ``only if part" of our main theorem, we have to give a direct construction for a 3-pyramidal STS$(v)$
whenever $v$ is admissible and not forbidden by Theorem \ref{non-existence}, hence for
any $v\equiv7, 9, 15$ (mod 24) and for any $v\equiv 3, 19$ (mod 48).
The most laboured constructions are those for the last two cases where we will use
{\it extended Skolem sequences} and {\it extended Langford sequences}.

\vskip0.3truecm
\noindent
Given a pair $(k,n)$ of positive integers with $1\le k\le 2n+1$, a  $k$-{\it extended Skolem sequence} of order $n$
can be viewed as a sequence  $(s_1,\dots, s_n)$ of $n$ integers such that

$$\bigcup_{i=1}^n \{s_i, s_i+i \} = \{1,2,\dots, 2n+1\}\setminus\{k\}.$$

\noindent
The existence question for extended Skolem sequences was completely settled by C. Baker in \cite{Ba}.

\begin{thm}\label{skolem} {\rm \cite{Ba}}
There exists a $k$-extended Skolem sequence of order $n$
if and only if either $k$ is odd and $n\equiv 0,1$ $($mod $4)$ or $k$ is even
and $n\equiv 2,3$ $($mod $4)$.
\end{thm}

A $k$-extended Langford sequence of order $n$
and defect $d$ can be viewed as a sequence of $n$ integers
$(\ell_1, ..., \ell_n)$ such that

$$\bigcup_{i=1}^n \{\ell_i, \ell_i+i+d-1 \} = \{1,2,\dots, 2n+1\}\setminus\{k\}.$$

\noindent
We need the following partial result about extended Langford sequences by V. Linek and S. Mor.

\begin{thm}\label{langford} {\rm\cite{LinMor}} There exists a $k$-extended Langford sequence of order $n$ and defect
$d$ for any triple $(k,n,d)$ with $n\ge 2d$, $n\equiv 2$ $($mod $4)$ and $k$ even.
\end{thm}

For general background concerning Skolem sequences, their variants, and their applications, we refer to
 \cite{S}. We also refer to the recent survey \cite{FM} which, however, fails to mention some of our work; for instance,
extended Skolem sequences have been crucial in our mentioned work on 1-rotational STSs \cite{BBRT,B} and
also in two papers \cite{BGCOM,WB} dealing, more generally, with 1-rotational $k$-cycle systems.

In the next lemma we combine Skolem sequences and Langford sequences to get the last ingredient that
we need for proving the ``only if part" of our main result. This lemma will be used, specifically, in the construction of a
3-pyramidal STS$(v)$ with $v\equiv3$ (mod 96).

\begin{lem}\label{skolem+langford}
There exists a $(\Z_{12n},\{3,4\},3,1)$-DF for any even $n\geq2$.
\end{lem}

\begin{proof}\quad We have to construct a set ${\cal F}_n$
of $2n-1$ triples with elements in $\Z_{12n}$
whose differences cover $\Bbb Z_{12n}\setminus \{0,3n,4n,6n,8n,9n\}$
exactly once. For the small cases $n\in\{2,4,6,8,12,14,20\}$ one can check that we can
take ${\cal F}_n=\{B_1,\dots,B_{2n-1}\}$ with $B_i=\{0,i,b_i\}$ and the $b_i$s as in the
following table:

\medskip
\scriptsize\noindent
\noindent\begin{tabular}{|l|c|r|c|r|c|r|c|r|c|r|c|r|}
\hline {$n$} & $(b_1,b_2,\dots,b_{2n-1})$     \\
\hline $2$ & $(5,9,13)$ \\
\hline $4$ & $(40,37,34,30,38,29,28)$ \\
\hline $6$ & $(13, 16, 20, 19, 26, 28, 30, 39, 34, 37, 40)$ \\
\hline $8$ & $(17, 20, 22, 25, 28, 33, 36, 34, 39, 47, 46, 52, 54, 45, 53)$ \\
\hline $12$ & $(66, 63, 37, 64, 38, 62, 39, 58, 40, 59, 41, 67, 42, 68, 43, 69, 44, 70, 45, 71, 46, 57, 47)$\\
\hline $14$ & $(76, 74, 43, 70, 44, 77, 45, 73, 46, 68, 47, 69, 48, 78,49, 79, 50, 80, 51, 81, 52, 82, 53, 83, 54, 67, 55)$\\
\hline $20$ & $(110, 103, 61, 108, 62, 102, 63, 100, 64, 105, 65, 106, 66, 107, 67, 98, 68,99, 69,111,70, 112, 71,$\\
& \hfill$113,72, 114, 73, 115, 74, 116, 75, 117, 76, 118, 77, 119, \
78, 97, 79)$ \\
\hline
\end{tabular}

\medskip\normalsize
For all the other values of $n$, set $n=2m$ and $\epsilon=r+(-1)^r$ where $r$
is the remainder of the Euclidean division of $m$ by 4.

By applying Theorem \ref{skolem} one can see that there exists a $(2m+1)$-extended
Skolem sequence $(s_1,\dots,s_{m+\epsilon})$ of order $m+\epsilon$. Also, by applying Theorem \ref{langford} one can see that
there exists a $(2m-2\epsilon)$-extended Langford sequence $(\ell_1,\dots,\ell_{3m-\epsilon-1})$ of order $3m-\epsilon-1$ and defect $m+\epsilon+1$.
Then one can see that the desired difference family is the one whose blocks are
the following:

\smallskip
$\{0, \ -i, \ s_i+4m-1\}$\hfill with $1 \le i \le m+\epsilon$;

\smallskip
$\{0, \ -(m+\epsilon+i), \ \ell_i+6m+2\epsilon\}$\hfill with $1 \le i \le 3m-\epsilon-1$.

\smallskip
\end{proof}

\begin{thm}\label{existence}
There exists a $3$-pyramidal STS$(v)$ for any admissible value of $v$
not forbidden by Theorem {\rm\ref{non-existence}}.
\end{thm}
\begin{proof}\quad Considering that the admissible values of $v$ are those congruent to 1 or 3 (mod 6),
we have to prove that there exists a $3$-pyramidal STS$(v)$ in the following five cases:
\begin{itemize}
\item[$(i)$] $v\equiv 7$ (mod $24$);
\item[$(ii)$] $v\equiv 15$ (mod $24$);
\item[$(iii)$] $v\equiv 9$ (mod $24$);
\item[$(iv)$] $v\equiv 3$ (mod $48$);
\item[$(v)$] $v\equiv 19$ (mod $48$).
\end{itemize}

Cases (i)-(ii). Assume that $v\equiv 7,15$ (mod $24$).
Let $\cal F$ be an $(H,\{h\},3,1)$-DF with $H=\Z_{6n+1}$ and $h=1$ if $v=24n+7$,
or $H=\Z_3\times\Z_{2n+1}$ and $h=3$ if $v=24n+15$.
In the former case the existence of $\cal F$ is guaranteed by Peltesohn's result (see Remark \ref{rem})
while in the latter it is enough to take ${\cal F}=\{\{(0,0),(1,i),(1,-i)\} \ | \ 1\leq i\leq n\}$.

The group $G:=\Bbb Z_2 \times \Bbb Z_2 \times H$ has order $v-3$ and exactly 3 involutions.
Thus, by Theorem \ref{DF}, it is enough to exhibit a $(G,\{2^3,3^e\},3,1)$-DF  for some suitable $e$.
Such a DF is, for instance, the one whose blocks are:

\smallskip
$\{(0,0)\}\times B$ \hfill with $B\in {\cal F}$;

\smallskip
$\{(0,1,0),(1,0,h),(1,1,-h)\}$\quad \hfill with $h\in \overline H$

\smallskip\noindent
where $\overline H$ is a subset of $H$ such that $\{\{h,-h\} \ | \ h\in \overline H\}$ is the {\it patterned starter} of $H$, i.e.,
the set of all symmetric 2-subsets of $H$ (see \cite{Dinitz}).

\bigskip
Case (iii). Assume that $v\equiv 9$ (mod $24$), say  $v=24n+9$. The group  $G:=\mathbb{D}_6\times \Z_{4n+1}$ has order $v-3$ and clearly has exactly three involutions.
Thus, by Theorem \ref{DF}, it is enough to exhibit a $(G,\{2^3,3^1\},3,1)$-DF (of course here the DF will ``use" multiplication on the first component and addition on the second).
For this, we have to distinguish three cases
according to whether $n$ is odd or congruent to 0 or 2 (mod 4). The three cases are very similar;
the blocks of the desired difference family can be taken as indicated in the table below.

\small
\begin{center}
\begin{tabular}{|l|c|r|c|r|c|r|c|r|c|r|c|r|}
\hline {Blocks} & $n$ odd & $n\equiv0$ (mod 4) & $n\equiv2$ (mod 4)      \\
\hline $\{(1,0),(x,n),(x,2n)\}$ & Yes & Yes & Yes   \\
\hline $\{(y,0),(1,- {n+1\over2}),(1,{3n+1\over2})\}$ & Yes & No & No \\
\hline $\{(1,0),(x,-{n\over2}),(x,{3n\over2})\}$ & No & Yes  & Yes \\
\hline $\{(y,0),(1,i),(1,2n+1-i)\}$ & $1 \le i \le n$ & $1 \le i \le n$ & $1 \le i \le n$\\
 & $i\ne {n+1\over2}$ & & \\
\hline $\{(y,0),(x,i),(x^2,-i)\}$ & $1\le i \le n$ & $1\le i \le n$, & $1\le i \le n$,  \\
 & &  $i\ne {n\over4}$ & and $i= {7n+2\over4}$ \\
\hline $\{(y,0),(x,-i),(x^2,i)\}$ & $n+1\le i \le 2n$ & $n+1\le i \le 2n$ & $n+1\le i \le 2n$ \\
& & and $i={n\over4}$ & and  $i\neq {7n+2\over4}$ \\
\hline $\{(1,0),(x,i),(x,2n-i)\}$ &  $1 \le i \le n-1$ &  $1 \le i \le n-1$ &  $1 \le i \le n-1$\\
& & and $i\ne {n\over2}$ & and $i\ne {n\over2}$ \\
\hline
\end{tabular}
\end{center}

\medskip\normalsize
We note that the subgroup of order 3 which is not covered by the differences of the above families is, in
any subcase, $\{(1,0),(x,0),(x^2,0)\}$.

\bigskip\noindent
Case (iv). Assume that $v\equiv3$ (mod 48), say  $v=48n+3$. The group $G=\Z_4\times\Z_{12n}$ has order $v-3$ and it has exactly three involutions.
Then, by Theorem \ref{DF}, it is enough to exhibit a $(G,\{2^3,3\},3,1)$-DF.

\medskip\noindent
Subcase (iv.1): $n$ is odd.

Take a $(2n+1)$-extended Skolem sequence $(s_1,...,s_{2n-1})$ of order $2n-1$ (which exists by Theorem \ref{skolem}).
Set $n=2t+1$ and check that the blocks of a $(G,\{2^3,3\},3,1)$-DF are the following:

$\{(0,0),(1,0),(3,6t+3)\};$

\smallskip
$\{(0,0),(1,3t+2),(1,-9t-5)\}$;

\smallskip
$\{(0,0),(1,i),(3,12t+7-i)\}$\hfill with $1 \le i \le 6t+3$ and $i\ne 3t+2;$

\smallskip
$\{(0,0),(1,6t+3+i),(3,6t+3-i)\}$\hfill with $1\le i \le  6t+2;$

\smallskip
$\{(0,0),(0,i),(0,-s_i-4t-1)\}$\hfill with $1\le i \le 4t+1$.

\medskip\noindent
Subcase (iv.2): $n$ is even.

Take a $(\Z_{12n},\{3,4\},3,1)$-DF $\cal F$ using Lemma \ref{skolem+langford}.
One can see that a $(G,\{2^3,3\},3,1)$-DF
is the one whose blocks are the following.

\smallskip
$\{(0,0),(1,0),(1,9n)\};$

\smallskip
$\{0\}\times B$ \hfill with $B\in{\cal F}$;

\smallskip
$\{(0,0),(1,i),(3,6n+1-i)\}$  \hfill with $1 \le i \le 3n$;

\smallskip
$\{(0,0),(1,i),(3,6n-i)\}$ \hfill with $3n+1 \le i \le 6n-1$.

\bigskip\noindent
Case (v). Assume that $v\equiv19$ (mod 48), say $v=48n+19$.
The group $G=\Z_4\times\Z_{12n+4}$ has order $v-3$ and it has exactly three involutions.
Then, by Theorem \ref{DF}, it is enough to exhibit a $(G,\{2^3\},3,1)$-DF.
Let $(s_1,s_2,...,s_{2n})$ be any $(n+1)$-extended Skolem sequence of
order $2n$ (which exists by Theorem \ref{skolem}).
Then the required difference family is the one whose blocks are the following:

\smallskip
$\{(0,0),(1,0),(1,3n+1)\}$;

\smallskip
$\{(0,0),(0,i),(0,-s_i-2n)\}$\hfill with $1 \leq i\leq 2n$;

\smallskip
$\{(0,0),(1,i),(3,6n+2-i)\}$\hfill wtih $1 \leq i \leq 3n$;

\smallskip
$\{(0,0),(1,6n+3-i),(3,i)\}$\hfill with $1 \leq i \leq 3n+1$.

\end{proof}

\section{Conclusion}

Theorem \ref{non-existence} and Theorem \ref{existence} are the ``if part" and the ``only if part" of
the main result Theorem \ref{main} which therefore is now completely proved.

The existence of a 3-pyramidal STS$(v)$ can be summarized in the following table
where, in the third column, we put a group acting 3-pyramidally on a STS$(v)$.

\medskip
\begin{center}
\begin{tabular}{|l|c|r|c|r|c|r|c|r|c|r|c|r|}
\hline {$v$} & Existence & Group     \\
\hline $24n+1$ & No & $-$   \\
\hline $24n+3$ & Yes $\Longleftrightarrow$ $n$ is even & $\Z_4\times\Z_{6n}$\\
\hline $24n+7$ & Yes & $\Z_2^2\times\Z_{6n+1}$\\
\hline $24n+9$ & Yes & $\mathbb{D}_6\times\Z_{4n+1}$\\
\hline $24n+13$ & No & $-$  \\
\hline $24n+15$ & Yes & $\Z_2^2\times\Z_3\times\Z_{2n+1}$\\
\hline $24n+19$ & Yes $\Longleftrightarrow$ $n$ is even & $\Z_4\times\Z_{6n+4}$\\
\hline $24n+21$ & No & $-$  \\
\hline
\end{tabular}
\end{center}

Note that an abelian group of order $24n+6$ has only one involution
so that there is no STS$(24n+9)$ which is 3-pyramidal under an abelian group.
Thus we see, from the above table, that there exists a STS$(v)$ which is 3-pyramidal under
an abelian group if and only if $v\equiv7, 15$ (mod 24) or $v\equiv 3, 19$ (mod 48).

%%%%%%%%%%%%%%%%%%%%%%%%%%%%%%%%%%%%%%%%%%%%%%%%%%%%%%
%\vskip0.3truecm
%\noindent
% Here is an example of $3-$pyramidal
%STS$(51)$ under the action of $\Bbb Z_4\times \Bbb Z_{12}$:

%\noindent
%The base blocks are:

%\noindent
%$$\{(\infty_1,\infty_2,\infty_3)\} \ \ \{\infty_1, (0,0),(0,6)\} \ \ \{\infty_2, (0,0),(2,0)\} \ \ \{\infty_3, (0,0),(2,6)\}$$

%\noindent
%$$\{(0,0),(0,4),(0,8)\}\ \ \{(0,0),(1,9),(3,4)\} \ \ \{(0,0),(1,11),(1,6)\} \ \ \{(0,0),(1,2),(1,5)\}$$

%%\noindent
%$$\{(0,0),(1,3),(1,4)\} \ \ \{(0,0),(1,0),(3,2)\} \ \ \{(0,0),(1,1),(3,5)\} \ \ \{(0,0),(2,1),(2,3)\}$$

%%%%%%%%%%%%%%%%%%%%%%%%%%%%%%%%%

\end{document}